\documentclass[preprint,12pt]{article}

 \usepackage{graphics}
 \usepackage{graphicx}
 \usepackage{epsfig}
 \usepackage[all]{xy}

\usepackage{amssymb}
\usepackage{amsthm,amsmath}

\usepackage{latexsym, epsfig, psfrag,eepic,colordvi,bm}
\usepackage{graphics,color}
\usepackage{amsmath,amsfonts,amssymb,amscd}
\usepackage[all]{xy}
\usepackage{epstopdf}

\usepackage{amsthm,amsmath,amssymb}

\usepackage{graphicx}

\theoremstyle{plain}
\newtheorem{theorem}{Theorem}[section]
\newtheorem{lemma}[theorem]{Lemma}
\newtheorem{corollary}[theorem]{Corollary}
\newtheorem{proposition}[theorem]{Proposition}

\theoremstyle{definition}
\newtheorem{definition}[theorem]{Definition}

\theoremstyle{remark}

\begin{document}

		

\title{On the local genus distribution of graph embeddings}

\author{Ricky X. F. Chen$^a$, Christian M. Reidys$^b$\\
	\small Virginia Bioinformatics Institute and Dept. of Mathematics,\\[-0.8ex]
	\small Virginia Tech, 1015 Life Sciences Circle,\\[-0.8ex]
	\small Blacksburg, VA 24061, USA\\
	\small\tt $^a$chen.ricky1982@gmail.com, $^b$duck@santafe.edu
}

\date{}
\maketitle


\begin{abstract}
The $2$-cell embeddings of graphs on closed surfaces have been widely studied.
It is well known that ($2$-cell) embedding a given graph $G$ on a closed
orientable surface is equivalent to cyclically ordering
the edges incident to each vertex of $G$. In this paper, we study the following problem:
given a genus $g$ embedding $\epsilon$ of the graph $G$ and a vertex of $G$,
how many different ways of reembedding the vertex such that the resulting embedding
$\epsilon'$ is of genus $g+\Delta g$?
We give formulas to compute this quantity and the local minimal genus achieved by reembedding.
In the process we obtain miscellaneous results. In particular,
if there exists a one-face embedding of $G$, then the probability of a random embedding of $G$
to be one-face is at least $\prod_{\nu\in V(G)}\frac{2}{deg(\nu)+2}$, where
$deg(\nu)$ denotes the vertex degree of $\nu$.
Furthermore we obtain an easy-to-check necessary condition for a given embedding of $G$ to
be an embedding of minimum genus.

  

\end{abstract}



\section{Introduction}

Graph embedding is one of the most important topics in topological graph theory.
In particular, $2$-cell embeddings of graphs (loops and multiple edges allowed) have been widely studied.
A $2$-cell embedding or map of a given graph $G$ on a closed surface of genus $g$, $S_g$,
is an embedding on $S_g$ such that the complement of any face is homeomorphic to an open disk.
The closed surfaces could be either orientable or unorientable. In this paper,
we restrict ourselves to the orientable case.

Let $g_{min}(G)$ and $g_{max}(G)$ denote the minimum and the maximum genus $g$ of the embeddings of
$G$, respectively.
There are many studies on determining these quantities and related problems~
\cite{duke,furst,liu2,jung,martin,liu,lzv,mohar,nebe,white,youngs,thoma,tesar,thoma2,liu3,xuong2}.
Assume $G$ has $e$ edges and $v$ vertices, and that $G$ is embedded in $S_g$ via $\epsilon$.
In view of Euler's characteristic formula,
\begin{align}
v-e+f=2-2g \quad \Longleftrightarrow \quad 2g=\beta(G)+1-f,
\end{align}
where $f\geq 1$ is the number of faces of $\epsilon$ and $\beta(G)$ is the Betti number of $G$. Thus,
the largest possible value of $g$ is $\lfloor\frac{\beta(G)}{2}\rfloor$. 

Any embedding of $G$ in a closed orientable surface can be equivalently represented by a
fatgraph, generated by $G$~\cite{edmonds,youngs}. A fatgraph generated by $G$ is the graph
$G$ with a specified cyclic order of edges around (i.e., incident to) each vertex of $G$,
i.e.~the topological properties of the embedding are implied in the cyclic orderings of edges.
Any variation of the local topological structure around a vertex, i.e., the cyclic order of edges
around the vertex, may change the topological properties of the whole embedding, e.g.~the genus of
the embedding.

Plane permutations \cite{chr-1} were recently used to study hypermaps, in particular to enumerate
hypermaps with one face. Since maps are particular hypermaps, we shall
employ plane permutations in order to study graph embeddings.

The paper is organized as follows: in Section $2$, we establish some basic facts on hypermaps.
In Section $3$, we study embeddings with one face.
We ask which local embeddings (reembeddings) of a fixed vertex do not affect the topological genus.
By changing the local embedding, we mean changing the cyclic order of edges around the vertex.
We shall show that the probability of preserving the genus
is at least $\frac{2}{deg(\nu)+2}$ for reembedding any vertex $\nu$ of degree (i.e., valence) $deg(\nu)$.
In addition we show that there exists at least one alternative way to reembed a vertex $\nu$
preserving the genus if $deg(\nu)\geq 4$.

In Section~$4$, we study embeddings with more than one face:
given a genus $g$ embedding $\epsilon$ of the graph $G$ and a vertex of $G$,
how many different ways of reembedding the vertex such that the resulting embedding
$\epsilon'$ is of genus $g+\Delta g$?
We give a formula to compute this quantity and the local minimal genus achieved by reembedding,
as well as an easy-to-check necessary condition for an embedding of $G$ to be
of minimum genus.


\section{Maps, hypermaps and plane permutations}

Let $\mathcal{S}_n$ denote the group of permutations, i.e.~the group
of bijections from $[n]=\{1,\dots,n\}$ to $[n]$, where the multiplication
is the composition of maps.
We shall use the following two representations of a permutation $\pi$:\\
\emph{two-line form:} the top line lists all elements in $[n]$, following the natural order.
The bottom line lists the corresponding images of elements on the top line, i.e.
\begin{eqnarray*}
	\pi=\left(\begin{array}{ccccccc}
		1&2& 3&\cdots &n-2&{n-1}&n\\
		\pi(1)&\pi(2)&\pi(3)&\cdots &\pi({n-2}) &\pi({n-1})&\pi(n)
	\end{array}\right).
\end{eqnarray*}
\emph{cycle form:} regarding $\langle \pi\rangle$ as a cyclic group, we represent $\pi$ by its
collection of orbits (cycles).
The set consisting of the lengths of these disjoint cycles is called the cycle-type of $\pi$.
We can encode this set into a non-increasing integer sequence $\lambda=\lambda_1
\lambda_2\cdots$, where $\sum_i \lambda_i=n$, or as $1^{a_1}2^{a_2}\cdots n^{a_n}$, where
we have $a_i$ cycles of length $i$. A cycle of length $k$ will be called a $k$-cycle.

For a permutation $\pi$ on $[n]$, we denote its total number of cycles by $C(\pi)$,
and we denote $Par_{\pi}$ the partition of $[n]$ induced by the cycles of $\pi$, i.e.,
every set of elements in a same cycle of $\pi$ contributes a part (or block) in $Par_{\pi}$.

A map or fatgraph having $n$ edges is a triple of permutations $(\alpha,\beta,\gamma)$ on
$[2n]$ where $\alpha$ is a fixed-point free involution and $\gamma=\alpha\beta$.
This can be seen as follows: we denote the two ends of an edge as half edges and
label the half edges using the labels from the set $[2n]$ so that each label appears exactly
once. This induces two permutations $\alpha$ and $\beta$, where $\alpha$ is a fixed point
free involution, whose cycles consist of the labels of the two half edges of the same
(untwisted) edge and $\beta$-cycles represent the counterclockwise cyclic arrangement of
all half edges incident to the same vertex. $\gamma=\alpha\beta$-cycles are called the
boundaries components.  The topological genus of a map $(\alpha,\beta,\gamma)$ satisfies
\begin{align}
C(\beta)-C(\alpha)+C(\gamma)=2-2g.
\end{align}
Any map induces a unique graph $G$, via $(\alpha,Par_{\beta})$, where each block
corresponds to a $G$-vertex and each $\alpha$-cycle determines a $G$-edge.

Hypermaps represent a generalization of maps by allowing hyper-edges, i.e., triples
$(\alpha, \beta, \gamma)$, where $\gamma=\alpha\beta$ and $\alpha$ is not necessarily
fixed point free. We can also define genus $g$ of a hypermap $(\alpha,\beta,\gamma)$ on
$[n]$, e.g., \cite{walsh2}, by
\begin{align}
C(\alpha)+C(\beta)+C(\gamma)-n=2-2g.
\end{align}
We will also call the pair $G=(\alpha, Par_{\beta})$ the underline graph of the hypermap
although it does not induce a conventional graph. Furthermore, any hypermap
$(\alpha',\beta',\gamma')$ having $G$ as the underlying graph is called an embedding of $G$.

\begin{definition}[Cyclic plane permutation]
	A cyclic plane permutation on $[n]$ is a pair $\mathfrak{p}=(s,\pi)$ where
	$s=(s_i)_{i=0}^{n-1}$ is an $n$-cycle and $\pi$ is an arbitrary permutation on $[n]$.
	The permutation $D_{\mathfrak{p}}=s\circ \pi^{-1}$ is called the diagonal of $\mathfrak{p}$.
\end{definition}\label{2def1}

Given $s=(s_0s_1\cdots s_{n-1})$, a cyclic plane permutation $\mathfrak{p}=(s,\pi)$ can be
represented by two aligned rows:
\begin{equation}
(s,\pi)=\left(\begin{array}{ccccc}
s_0&s_1&\cdots &s_{n-2}&s_{n-1}\\
\pi(s_0)&\pi(s_1)&\cdots &\pi(s_{n-2}) &\pi(s_{n-1})
\end{array}\right).
\end{equation}
Indeed, $D_{\mathfrak{p}}$ is determined by the diagonal-pairs (cyclically) in the two-line
representation here, i.e., $D_{\mathfrak{p}}(\pi(s_{i-1}))=s(s_{i-1})=s_i$ for $0<i< n$, and
$D_{\mathfrak{p}}(\pi(s_{n-1}))=s_0$.
W.l.o.g.~we shall assume $s_0=1$ and by ``the cycles of $\mathfrak{p}=(s,\pi)$'' mean the
cycles of $\pi$. We will refer to the blocks in $Par_{\pi}$ as $\mathfrak{p}$-vertices or
vertices for short, and elements in a vertex half edges.

Hypermaps having one face, $(\alpha,\beta,\gamma)$, can be represented as cyclic plane
permutations $(\gamma,\beta)$. In particular, one-face maps are cyclic
plane permutations where the diagonals are fixed-point free involutions.
Any cyclic plane permutation having a fixed-point free involution diagonal encodes a
one-face embedding of some graph.

A fatgraph with one boundary component is displayed in Figure~\ref{fig-fig1}.
The two presentations are showing the same one-face map, whose corresponding
cyclic plane permutation reads
\begin{eqnarray*}
	\mathfrak{p}=\left(\begin{array}{cccccccc}
		1&2&3&4&5&6&7&8\\
		1&6&7&8&3&4&5&2
	\end{array}\right).
\end{eqnarray*}

\begin{figure}[!htb]
	\centering
	\includegraphics[scale=.6]{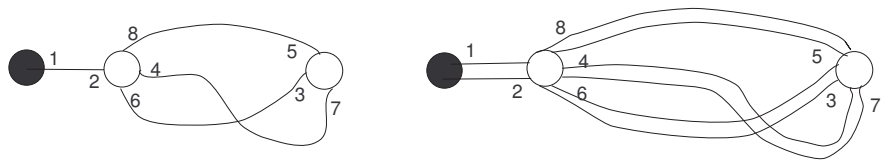}
	{\centering \caption{A one-face map with $4$ edges.}\label{fig-fig1}}
	
\end{figure}

\begin{definition}[Localization]
	Given a cyclic plane permutation $\mathfrak{p}=(s,\pi)$ on $[n]$ and a $\mathfrak{p}$-vertex
	$\nu$, the localization of $\mathfrak{p}$ at $\nu$, $\text{\rm loc}_\nu(\mathfrak{p})=
	(s_{\nu},\pi_{\nu})$ is the cyclic plane permutation
	\begin{equation*}
	(s_{\nu},\pi_{\nu})=
	\left(\begin{array}{ccccc}
	s_{i1}&s_{i2}&\cdots &s_{i(k-1)}&s_{ik}\\
	\pi(s_{i1})&\pi(s_{i2})&\cdots &\pi(s_{i(k-1)}) &\pi(s_{ik})
	\end{array}\right),
	\end{equation*}
	which is obtained by deleting all columns not containing half edges incident to $\nu$
	in the two-line representation of $\mathfrak{p}=(s,\pi)$.
\end{definition}

Let $D_{\nu}$ denote the diagonal of $\text{\rm loc}_{\nu}(\mathfrak{p})$, i.e.,
$D_\nu=s_\nu\circ \pi_\nu^{-1}$. Note that even
if $(s,\pi)$ is a map, $D_\nu$ is not necessarily a fixed-point free involution. For example,
given $(s,\pi)$
$$
\left(\begin{array}{cccccccccccc}
1 & 3 &2 &5 &7 &4 &6 &9 &8 &10 &11 & 12\\
5 & 8 &3 &4 &7 &10 &12 &2 &6 &1 &9 & 11
\end{array}\right),
$$
where $\pi=(1,5,4,10)(2,3,8,6,12,11,9)(7)$,
let $\nu=\{2,3,8,6,12,9,11\}$. Then,
$$
\text{\rm loc}_\nu(\mathfrak{p})=
\left(\begin{array}{ccccccc}
3 & 2 &6 &9 &8 &11 &12\\
8 & 3 &12 &2&6 &9 &11
\end{array}\right),
$$
we arrive at
$
D_{\nu}=(2,8)(3,6,11)(9,12).
$

A set of consecutive diagonal-pairs in
$\mathfrak{p}=(s,\pi)$ is called a diagonal
block. In above example,
$
\begin{array}{ccccc}
&2 &5 &7 &4 \\
8 &3 &4 &7 &
\end{array}
$
is a diagonal block. It is completely determined by its corners, in this case, the
lower left corner, $8$, as well as the upper right corner, $4$. The diagonal block
is denoted by $<8,4>$.

Given a cyclic plane permutation $\mathfrak{p}=(s,\pi)$ on $[n]$ and a sequence
$h=h_1h_2\cdots h_{n-1}$ on $[n-1]$, let $s^h=(s_0, s_{h_1}, s_{h_2}, \ldots s_{h_{n-1}})$, i.e.~$s$ is acted upon by $h$ via translation its
indices,
and $\pi^h=D_{\mathfrak{p}}^{-1} \circ s^h$.  This induces the new cyclic plane permutation $(s^h,\pi^h)$ having by
construction the same diagonal as $(s,\pi)$.
Equivalently, the two-line representation of $(s^h,\pi^h)$ can be obtained by
permuting the diagonal-pairs of $(s,\pi)$. In the following, a cyclic plane
permutation written in the form like $(s^h,\pi^h)$, always means that it is obtained
from $(s,\pi)$ by permuting diagonal-pairs by $h$.

\begin{lemma}[Localization lemma]\label{2lem2}
	Let $(s,\pi), (s',\pi')=(s^H,\pi^H)$ be cyclic plane permutations
	such that $Par_{\pi}=Par_{\pi'}$ and $\pi$ and $\pi'$ exclusively
	differ at the vertex $\nu$.
	Then, there exists some $h$ such that $(s'_{\nu},\pi'_{\nu})=(s_{\nu}^h,\pi_{\nu}^h)$
	and furthermore $(D_{\nu},Par_{\pi_{\nu}})=(D'_{\nu},Par_{\pi'_{\nu}})$.
\end{lemma}
\begin{proof}
	Assume $\mathfrak{p}=(s,\pi)$ and $\mathfrak{p'}=(s', \pi')$ are respectively
			
				\begin{eqnarray*}
					& \left(
					\vcenter{\xymatrix@C=0pc@R=1pc{
							\cdots s_{i_0} & s_{i_0+1} &{\cdots} & s_{i_1}\ar@{--}[dl] & \cdots & s_{i_{k-2}} &{\cdots} & s_{i_{k-1}}\quad \ar@{--}[dl] \cdots\\
							\cdots \pi(s_{i_0})\ar@{--}[ur] & {\cdots} & \pi(s_{i_1-1}) & \pi(s_{i_1}) & \cdots & \pi(s_{i_{k-2}})\ar@{--}[ur] & {\quad \cdots \quad} & \pi(s_{i_{k-1}}) \cdots
						}}
						\right),\\
						& \left(
						\vcenter{\xymatrix@C=0pc@R=1pc{
								\cdots s'_{i_0} & s'_{i_0+1} &{\cdots} & s'_{i_1}\ar@{--}[dl] & \cdots & s'_{i_{k-2}} &{\cdots} & s'_{i_{k-1}}\quad \ar@{--}[dl] \cdots\\
								\cdots \pi'(s'_{i_0})\ar@{--}[ur] & {\cdots} & \pi'(s'_{i_1-1}) & \pi'(s'_{i_1}) & \cdots & \pi'(s'_{i_{k-2}})\ar@{--}[ur] & {\quad \cdots \quad} & \pi'(s'_{i_{k-1}}) \cdots
							}}
							\right),
						\end{eqnarray*}
			where we assume $\nu=\{s_{i_0},s_{i_1},\ldots s_{i_{k-1}}\}=
			\{s'_{i_0},s'_{i_1},\ldots s'_{i_{k-1}}\}$ and $s'_0=s_0$.
			Since by assumption $\pi$ and $\pi'$ only differ at the vertex $\nu$, we have $s_j=s'_j$
			for $0 \leq j \leq i_0$. Furthermore $Par_{\pi}=Par_{\pi'}$ implies
			$Par_{\pi_{\nu}}=Par_{\pi'_{\nu}}$.\\
			{\it Claim.} $D_{\nu}=D'_{\nu}$.\\
			To show this we observe that for fixed $j$, each diagonal block
			$<\pi'(s'_{i_j}),s'_{i_{j+1}}>$ equals the diagonal block
			$<\pi(s_{i_l}),s_{i_{l+1}}>$ for some $l(j)$, i.e.,
			\begin{eqnarray*}
				\left(
				\vcenter{\xymatrix@C=0pc@R=1pc{
						& s'_{i_j+1}\ar@{--}[dl] & s'_{i_j+2} &\cdots & s'_{i_{j+1}}\ar@{--}[dl]\\
						\pi'(s'_{i_j}) & \pi'(s'_{i_j+1})  &{\cdots} & \pi'(s'_{i_{j+1}-1}) &
					}}
					\right)=
					\left(
					\vcenter{\xymatrix@C=0pc@R=1pc{
							& s_{i_l+1}\ar@{--}[dl] & s_{i_l+2} &\cdots & s_{i_{l+1}}\ar@{--}[dl]\\
							\pi(s_{i_l}) & \pi(s_{i_l+1})  &{\cdots} & \pi(s_{i_{l+1}-1}) &
						}}
						\right)
					\end{eqnarray*}
					To prove this, we observe that by construction for fixed $j$ there exists some
					index $l$ such that $\pi(s_{i_l})=\pi'(s'_{i_j})$ holds.
					This implies,
					$$
					s'_{i_j+1}=D_{\mathfrak{p}}\circ \pi'(s'_{i_j})=D_{\mathfrak{p}}\circ \pi(s_{i_l})=s_{i_l+1}.
					$$
					In case of $s'_{i_j+1}\not\in \nu$, we have
					$\pi'(s'_{i_j+1})=\pi(s'_{i_j+1})=\pi(s_{i_l+1})$ and derive
					$$
					s'_{i_j+2}=D_{\mathfrak{p}}\circ \pi'(s'_{i_j+1})=D_{\mathfrak{p}}\circ \pi(s_{i_l+1})=s_{i_l+2}.
					$$
					Iterating we arrive at $s'_{i_{j+1}}=s_{i_{l+1}}$, whence the two diagonal blocks are equal,
					the Claim follows and the proof of the lemma is complete.
				\end{proof}

				\begin{definition}\label{2def3}
					Given a cyclic plane permutation $\mathfrak{p}=(s,\pi)$ on $[n]$ and its localization
					at $\nu$, $\text{\rm loc}_{\nu}(\mathfrak{p})=(s_v,\pi_v)$. Suppose $(s_{\nu}^h,\pi_{\nu}^h)$ is such that
					$Par_{\pi_{\nu}}=Par_{\pi_{\nu}^h}$. Then the inflation of $(s_{\nu}^h,\pi_{\nu}^h)$ w.r.t.
					$\mathfrak{p}$ is the cyclic plane permutation
					$
					\text{\rm inf}_{\mathfrak{p}}((s_\nu^h,\pi_{\nu}^h))
					$,
					obtained from $(s_{\nu}^h,\pi_{\nu}^h)$ by substituting each diagonal-pair with the diagonal
					block in $\mathfrak{p}$ having the diagonal-pair as its corners.
				\end{definition}
				
				Let
				$
				(s_{\nu}^h,\pi_{\nu}^h)=\left(\begin{array}{ccccccc}
				3 & 9 &8 &11 & 2 &6 &12\\
				12 &2&6 & 8 & 3 & 9 &11
				\end{array}\right),
				$
				then the inflation of $(s_{\nu}^h,\pi_{\nu}^h)$ w.r.t. $(s,\pi)$ is
				$$
				\text{\rm inf}_{\mathfrak{p}}((s_\nu^h,\pi_{\nu}^h))=\left(\begin{array}{cccccccccccc}
				1 & 3 &9 &8 &10 &11 &2 &5 &7 &4 &6 & 12\\
				5 & 12&2 &6 &1 &8 &3 &4 &7 &10 &9 & 11
				\end{array}\right).
				$$

				\begin{lemma}[Inflation lemma]\label{lem211}
					Let $ \text{\rm inf}_{\mathfrak{p}}((s_\nu^h,\pi_{\nu}^h))=(s',\pi')$. Then we have
					$$
					(D_{\mathfrak{p}},Par_{\pi})=(D_{\mathfrak{p'}},Par_{\pi'})
					$$
					and $\pi'$ differs from $\pi$ only at the vertex $\nu$.
				\end{lemma}
				\proof
				By construction, $D_\mathfrak{p'}=D_\mathfrak{p}$. Any half edge not contained in $\nu$, is
				located inside the respective diagonal blocks, whence $\pi$ and $\pi'$ are equal on these
				half edges. All $\nu$-half edges contribute exactly one block in $Par_{\pi}$ and $Par_{\pi'}$,
				since $Par_{\pi_{\nu}}=Par_{\pi_{\nu}^h}$. Accordingly, we have $Par_{\pi}=Par_{\pi'}$, completing
				the proof of the lemma. \endproof

				Combining Lemma~\ref{2lem2} and~\ref{lem211}, we obtain
				\begin{theorem}\label{2thm2}
					Let $\mathfrak{p}=(s,\pi)$ be a cyclic plane permutation.
					Let $X=\{H\mid  Par_{\pi}=Par_{\pi^H}\ \wedge \  \pi(i)=\pi^H(i), \ i\not\in \nu\}$ and
					let $Y=\{h\mid Par_{\pi_{\nu}}=Par_{\pi_{\nu}^h}\}$. Then there is a bijection between $X$ and $Y$
					and we have the commutative diagram:
					\[
					\xymatrixcolsep{10pc}\xymatrix{
						{(s,\pi)} \ar[d]^{\text{\rm loc}_\nu} \ar[r]^{H:\quad Par_{\pi}=Par_{\pi^H}} &{(s',\pi')}\\
						(s_{\nu},\pi_{\nu})\ar[r]_{h:\quad Par_{\pi_{\nu}}=Par_{\pi_{\nu}^h}} &
						(s_{\nu}^h,\pi_{\nu}^h)\ar[u]^{\text{\rm inf}_{\mathfrak{p}}}
					}
					\]
				\end{theorem}
				
				
				\section{Embeddings having one face}
				
				
				\begin{lemma}
					Let $\mathfrak{p}=(s,\pi)$ be a cyclic plane permutation with the underline graph
					$G=(D_{\mathfrak{p}},Par_{\pi})$. Then, $\mathfrak{p'}=(s',\pi')$ is an embedding of
					$G$ iff $(s',\pi')=(s^h,\pi^h)$ for some $h$ and $(D_{\mathfrak{p'}},Par_{\pi'})=
					(D_{\mathfrak{p}},Par_{\pi})$.
				\end{lemma}
				
				\proof If $\mathfrak{p'}=(s',\pi')$ is an embedding of $G=(D_{\mathfrak{p}},Par_{\pi})$,
				then, by definition, $(D_{\mathfrak{p'}},Par_{\pi'})=(D_{\mathfrak{p}},Par_{\pi})$.
				Thus, $D_{\mathfrak{p}}=D_{\mathfrak{p'}}=s'\circ \pi'^{-1}$. Clearly, there exists some $h$ such
				that $s'=s^h$. By construction we have $\pi^h=D_{\mathfrak{p}}^{-1}\circ s^h=
				D_{\mathfrak{p'}}^{-1}\circ s'=\pi'$. Hence, $(s',\pi')=(s^h,\pi^h)$ for some $h$.
				The converse is clear, whence the lemma.\qed
				
				This lemma shows that any one-face embedding is originated by the action of some $h$ on a
				cyclic plane permutation. Explicitly, by permuting diagonal-pairs based on a fixed
				one-face embedding $(s,\pi)$ of the graph such that $Par_{\pi}=Par_{\pi^h}$.
				
				\begin{corollary}
					Let $\mathfrak{p}=(s,\pi)$ be a one-face map with the underline graph $G=(D_{\mathfrak{p}}, Par_{\pi})$.
					Fixing the (local) embedding of all vertices of $G$ as in $\mathfrak{p}$
					but the vertex $\nu$, each local embedding of $\nu$ leading to a one-face map $\mathfrak{p'}=(s',\pi')$
					corresponds to a $H$ such that $(s',\pi')=(s^H,\pi^H)$ and $\pi'$ differs from $\pi$ only at the vertex $\nu$.
				\end{corollary}
				
				According to the bijection between $H$ and $h$ in Theorem~\ref{2thm2}, the number of different embeddings of $\nu$ keeping
				one face is equal to the number of different $h$ such that $(s_{\nu}^h,\pi_{\nu}^h)$ and
				$(s_{\nu},\pi_{\nu})$ having the same underlying graph.
				We denote this number by $R_{\nu}$. Moreover, the half edges
				contained in $\nu$ split the cyclic plane permutation into $|\nu|$ diagonal blocks.
				We can view these diagonal blocks as they are arranged in a circular fashion, as displayed
				in Figure~\ref{fig-circ}. To reembed $\nu$ means to permute these diagonal blocks circularly.
				\begin{figure}[!htb]
					\centering
					\includegraphics[scale=.6]{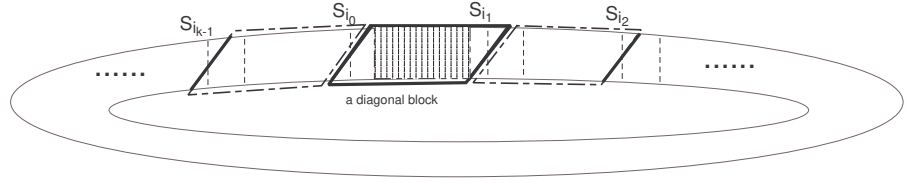}
					{\centering \caption{Circular arrangement of diagonal blocks determined by the
							vertex $v$}\label{fig-circ}}
					
				\end{figure}

				Let $U_{D}$ denote the set of cyclic plane permutations having diagonal $D$, where $D$ is a
				fixed permutation. Note $\mathfrak{p}=(s,\pi)\in U_D$ iff $D=D_{\mathfrak{p}}=s\circ \pi^{-1}$.
				Thus, $|U_D|$ enumerates the ways to write $D$ as a product of an $n$-cycle with
				another permutation.

				A partition $\lambda$ of $n$ is written as $\lambda\vdash n$.
				Let $\mu\vdash n,\eta\vdash n$. We write $\mu\rhd_{2i+1}\eta$ if $\mu$
				can be obtained by splitting one $\eta$-block into $(2i+1)$ non-zero parts. Let
				furthermore $\kappa_{\mu,\eta}$ denote the number of different ways to obtain
				$\eta$ from $\mu$ by merging $\ell(\mu)-\ell(\eta)+1$ $\mu$-blocks into one, where
				$\ell(\mu)$ and $\ell(\eta)$ denote the number of blocks in the partitions $\mu$ and $\eta$,
				respectively.
				\begin{theorem}\cite{chr-1}\label{2cor1}
					Let $p_{k}^{\lambda}(n)$ denote the number of $\mathfrak{p}\in U_{D}$ having
					$k$ cycles, where $D$ is of cycle-type $\lambda$. 
					Let $q^{\lambda}$ denote the number of permutations of cycle-type $\lambda$.
					Then, 
					\begin{equation}
					p_k^{\lambda}(n)=\frac{\sum_{i=1}^{\lfloor\frac{n-k}{2}\rfloor}{k+2i\choose k-1}p_{k+2i}^{\lambda}(n)q^{\lambda}
						+\sum_{i=1}^{\lfloor\frac{n-\ell(\lambda)}{2}\rfloor}\sum_{\mu\rhd_{2i+1}\lambda}
						\kappa_{\mu,\lambda}p_k^{\mu}(n)q^{\mu}}{q^{\lambda}[n+1-k-\ell(\lambda)]}.
					\end{equation}
				\end{theorem}
				
				Note that the localizations $(s_{\nu}^h,\pi_{\nu}^h)$ and $(s_{\nu},\pi_{\nu})$ have
				the same underlying hyper-graph $G$, iff both of them belong to $U_{D}$ where
				$D=D_{\nu}$ and $C(\pi_{\nu}^h)=C(\pi_{\nu})=1$.
				Therefore, if $D_{\nu}$ has cycle-type $\lambda$, $R_{\nu}=p_1^{\lambda}(|\nu|)$.
				As a ressult we obtain
				
				\begin{theorem}\label{3thm2}
					Let $\epsilon$ be a one-face embedding of $G$, and $\nu$ be a vertex of $G$ with
					$deg(\nu)\geq 4$. Then there exists at least one additional way to reembed $\nu$
					such that the obtained embedding $\epsilon'$ has the same genus as $\epsilon$.
				\end{theorem}
				\begin{proof} Assume $d\geq 4$ and $\epsilon$ is localized at $\nu$
					\begin{eqnarray*}
						(s_{\nu},\pi_{\nu})=\left(\begin{array}{ccccc}
							v_1&v_2&\cdots&v_{d-1}&v_d\\
							v_{i,1}&v_{i,2}&\cdots&v_{i,d-1}&v_{i,d}
						\end{array}\right),
					\end{eqnarray*}
					where $\pi_{\nu}=(v_1,V_2,\cdots, V_{d-1},V_d)$.
					Firstly, if $V_l=v_p, V_m=v_q$ and $1<l<m\leq d, 1<p<q\leq d$, i.e.~Case~$3$
					of~\cite[Lemma~$3$]{chr-1}, then there exists at least one additional way to
					reembed $\nu$ preserving genus.
					Otherwise, we have $\pi_{\nu}=(v_1,v_d,v_{d-1},\ldots,v_2)$.
					In this case,
					\begin{eqnarray*}
						(s_{\nu},\pi_{\nu})=\left(\begin{array}{ccccc}
							v_1&v_2&\cdots&v_{d-1}&v_d\\
							v_{d}&v_{1}&\cdots&v_{d-2}&v_{d-1}
						\end{array}\right),
					\end{eqnarray*}
					whence
					\begin{eqnarray*}
						D_{\nu}=\left\{
						\begin{array}{cc}
							(v_1,v_3,\ldots,v_d,v_2,v_4,\ldots,v_{d-1}), &\quad d\in odd,\\
							(v_1,v_3,\ldots,v_{d-1})(v_2,v_4,\ldots,v_d), &\quad d\in even.
						\end{array}
						\right.
					\end{eqnarray*}
					It remains to show that if $d\geq 4$ we have $R_{\nu}\geq 2$ in all cases.
					To this end, we apply a formula for $p_1^{\lambda}(k)$ due to Stanley~\cite{stan}.
					If $\lambda=(1^{a_1},2^{a_2},\ldots,k^{a_k})$, then
					\begin{equation}\label{3eq6}
					p_1^{\lambda}(k)=\sum_{i=0}^{k-1}\frac{i!(k-1-i)!}{k}\sum_{<r_1,\ldots,r_i>}
					{a_1-1\choose r_1}{a_2\choose r_2}\cdots{a_i\choose r_i}(-1)^{r_2+r_4+r_6+\cdots},
					\end{equation}
					where $<r_1,\ldots,r_i>$ ranges over all non-negative integer solutions of the
					equation $\sum_j jr_j=i$. Applying Stanley's formula we can compute $R_{\nu}$, if
					$d\in odd$ as
					\begin{eqnarray*}
						R_{\nu}=\frac{(d-1)!}{d}\sum_{i=0}^{d-1}(-1)^i{d-1\choose i}^{-1}=\frac{2(d-1)!}{d+1}.
					\end{eqnarray*}
					The simplification of the summation is stems from the following formula \cite{spru}
					\begin{eqnarray*}
						\sum_{i=0}^n (-1)^i{x\choose i}^{-1}=\frac{x+1}{x+2}(1+(-1)^n{x+1\choose n+1}^{-1}).
					\end{eqnarray*}
					It is not hard to see that $R_{\nu}\geq 2$ if $d\geq 4$. Similarly, if $4|d$ and
					$d\geq 4$, we have
					\begin{eqnarray*}
						R_{\nu}&=&\sum_{i=0}^{\frac{d}{2}-1}(-1)^i\frac{i!(d-1-i)!}{d}+
						\sum_{i=\frac{d}{2}}^{d-1}(-1)^i\frac{i!(d-1-i)!}{d}[(-1)^i+(-1)^{i-\frac{d}{2}}{2\choose 1}(-1)]\\
						&=&\frac{2(d-1)!}{d+1}(1-{d\choose \frac{d}{2}}^{-1}).
					\end{eqnarray*}
					If $d\in even$ and $4\nmid d$, we have
					\begin{eqnarray*}
						R_{\nu}&=&\sum_{i=0}^{\frac{d}{2}-1}(-1)^i\frac{i!(d-1-i)!}{d}+
						\sum_{i=\frac{d}{2}}^{d-1}(-1)^i\frac{i!(d-1-i)!}{d}[(-1)^i+(-1)^{i-\frac{d}{2}}{2\choose 1}]\\
						&=&\frac{2(d-1)!}{d+1}(1+{d\choose \frac{d}{2}}^{-1}).
					\end{eqnarray*}
					In both cases, if $d\geq 4$, it is straightforward to show that $R_{\nu}\geq 2$,
					since both $\frac{2(d-1)!}{d+1}$ and $(1-{d\choose \frac{d}{2}}^{-1})$
					are increasing functions of $d$.
					Accordingly, in all cases, if $d\geq 4$, then $R_{\nu}\geq 2$,
					completing the proof.
				\end{proof}
				
				Note in case of a vertex $\nu$ having degree $1$ or $2$, the situation is clear. Thus it
				remains to consider the case $deg(\nu)=3$. For such a vertex, a reembedding preserving genus
				can be impossible. For example, $D_{\nu}=(132)=(123)(312)$ is the unique decomposition of
				$D_{\nu}$.
				
				\begin{corollary}\label{3cor3}
					Any even permutation on $[n]$ with $n\geq 4$ has at least two different factorizations into
					two $n$-cycles.
				\end{corollary}
				\proof Since $D_{\nu}=s_{\nu}\circ \pi_{\nu}^{-1}$ and both $s_{\nu}$ as well as $\pi_{\nu}$
				have only one cycle, $D_{\nu}$ is an even permutation.
				Theorem~\ref{3thm2} implies that $D_{\nu}$ has at least $2$ factorizations into two
				$n$-cycles. \qed
				
				\begin{corollary}\label{3cor4}
					Let $G$ be a graph having $m$ vertices of degree no less than $4$.
					If there exists a one-face embedding of $G$, then there are at least
					$2^m$ one-face embeddings of $G$.
				\end{corollary}

				\begin{theorem}
					Let $(s,\pi)$ be a one-face embedding of $G$ and $(s_{\nu},\pi_{\nu})$ its localization
					at $\nu$. Suppose $D_{\nu}$ has cycle-type
					$\lambda=(1^{a_1},2^{a_2},\ldots,k^{a_k})$ where $k=deg(\nu)$, then the probability $prob_1(\nu)$ of a
					reembedding of $\nu$ to be one-face satisfies
					\begin{align}\label{3eq7}
					\frac{2}{deg(\nu)-a_1+2} \leq prob_1(\nu) \leq \frac{2}{deg(\nu)-a_1+\frac{19}{29}}.
					\end{align}
					In particular, for any vertex $\nu$, $prob_1(\nu)\geq \frac{2}{deg(\nu)+2}$.
				\end{theorem}
				\proof In Zagier~\cite{zag}, it was proved that
				$$
				\frac{2(k-1)!}{k-a_1+2} \leq p_1^{\lambda}(k) \leq \frac{2(k-1)!}{k-a_1+\frac{19}{29}}.
				$$
				Since there are $(k-1)!$ different ways to embed $\nu$, eq.~\eqref{3eq7} immediately
				follows in view of $p_1^{\lambda}(k)=R_{\nu}$. Clearly, we have
				$\frac{2}{deg(\nu)-a_1+2} \geq \frac{2}{deg(\nu)+2}$, whence the second assertion. \qed

				\begin{corollary}
					If there exists a one-face embedding of $G$, then the probability
					of a random embedding of $G$ to have one face is at least
					$\prod_{\nu \in V(G)} \frac{2}{deg(\nu)+2}$.
				\end{corollary}
				
				\section{Embeddings with multiple faces}
				
				In this section, we generalize cyclic plane permutations to general plane permutations.
				This puts us in position to study graph embeddings having $k$ faces. Although it is hard
				to determine $g_{min}$ and $g_{max}$, as well as the genus distribution for a given graph
				$G$, we will show that locally these quantities can be more easily obtained.
				
				\begin{definition}\label{4def1}
					A plane permutation on $[n]$ is a pair, $\mathfrak{p}$, of
					permutations $s$ and $\pi$ on $[n]$. The permutation $D_{\mathfrak{p}}=s\circ \pi^{-1}$
					is called the diagonal of $\mathfrak{p}$. If $s$ has $k$ cycles, we write $\mathfrak{p}=(s,\pi)_k$.
				\end{definition}
				
				Assume $s=(s_{11}, \ldots s_{1m_1})(s_{21},\ldots s_{2m_2})\cdots (s_{k1},\ldots s_{km_k})$,
				where $\sum_i m_i =n$. A plane permutation $(s,\pi)_k$ can be represented by two aligned rows:
				\begin{align*}
				\left(\begin{array}{ccccccccccc}
				\boxed{s_{11}} & s_{12} & \cdots & s_{1m_1} & \boxed{s_{21}}&\cdots &\quad s_{2m_2}\quad \cdots &\boxed{s_{k1}}& \cdots & s_{km_k}\\
				\pi(s_{11})&\pi(s_{12})&\cdots & {\boxed{\pi(s_{1m_1)}}}& \pi(s_{21}) & \cdots  & \boxed{\pi(s_{2m_2})}
				\cdots & \pi(s_{k1}) & \cdots & \boxed{\pi(s_{km_k})}
				\end{array}\right).
				\end{align*}
				$D_{\mathfrak{p}}$ can be defined as follows:
				\begin{itemize}
					\item  For $1 \leq i \leq k$, $D_{\mathfrak{p}}(\pi(s_{ij}))=s_{i(j+1)}$ if $j\neq m_i$;
					\item For $1 \leq i \leq k$, $D_{\mathfrak{p}}(\pi(s_{im_i}))=s_{i1}$.
				\end{itemize}
				We call blocks
				\begin{align*}
				\left(\begin{array}{cccc}
				\boxed{s_{i1}} & s_{i2} & \cdots & s_{im_i} \\
				\pi(s_{i1})&\pi(s_{i2})&\cdots & {\boxed{\pi(s_{im_i)}}}
				\end{array}\right)
				\end{align*}
				the cycles of the plane permutation. If the face
				$(s_{i1},\ldots, s_{im_i})$ is incident to a $\mathfrak{p}$-vertex $\nu$,
				the corresponding cycle is said to be incident to $\nu$.
				Since every embedding having $k$ faces can be
				represented by a triple $(\alpha, \beta, \gamma)$, where $\gamma=\alpha \beta$ and
				$\gamma$ has $k$ cycles, any embedding can be expressed via a plane permutation
				$(\gamma,\beta)_k$. Let $H(f)$ denote the set of half edges contained in the face $f$.
				
				\begin{lemma}\label{4lem1}
					Let $\nu$ be a vertex of the graph $G$ and $\epsilon$ be an embedding of $G$, where $\nu$
					is incident to $q$ faces, $f_i$, for $1\leq i \leq q$.
					Let $\epsilon'$ be an embedding, obtained by reembedding $\nu$ such that $\nu$ is
					incident to $q'$ faces, $f'_i$, for $1 \leq i \leq q'$.
					Then we have
					$$
					\bigcup_{i=1}^q H(f_i)=\bigcup_{i=1}^{q'} H(f'_i), \quad\quad q\equiv q' \pmod{2}.
					$$
				\end{lemma}
				\proof
				Let $\epsilon,\epsilon'$ be two embeddings represented by $\mathfrak{p}=(s,\pi)_k$ and
				$\mathfrak{p'}=(s',\pi')_{k'}$, respectively, such that $D_{\mathfrak{p}}=D_{\mathfrak{p'}}$
				and $Par_{\pi}=Par_{\pi'}$. Note that $\epsilon$ and $\epsilon'$ only differ w.r.t.~the
				cyclic order of the half edges around $\nu$. Thus, for $z\not\in \nu$, we have
				$\pi(z)=\pi'(z)$. Clearly, any face $f$ of $\epsilon$ can be expressed as the sequence
				$(D_{\mathfrak{p}}\pi(z), (D_{\mathfrak{p}}\pi)^2(z), \ldots)$ for any $z\in H(f)$.
				The lemma is implied by the following
				
				{\it Claim.} Any face $f'$ of $\epsilon'$ either intersects some $\epsilon$-face
				$f_i$ for $1\leq i \leq q$ and is incident to $\nu$ or it coincides with an
				$\epsilon$-face $f$ not incident to $\nu$.\\
				
				Suppose $f'$ does not intersect any $f_i$ for $1\leq i \leq q$. For any $z\in H(f')$,
				$z\not\in \nu$ holds and by construction $\pi'(z)=\pi(z)$. As a result,
				$D_{\mathfrak{p}}(\pi(z))=D_{\mathfrak{p}}(\pi'(z))=D_{\mathfrak{p'}}(\pi'(z))$,
				i.e.~$f'$ coincides with an $\epsilon$-face $f$ that is not incident to $\nu$.
				
				If $f'$ intersects some $\epsilon$-face $f_i$ for $1\leq i \leq q$, we shall prove that
				$f'$ is incident to $\nu$. Assume the half edge $u$ is contained in the $\epsilon$-face $f_j$ as
				well as in the face $f'$ of $\epsilon'$. Then,
				\begin{align*}
				f_j &=(D_{\mathfrak{p}}\pi(u), (D_{\mathfrak{p}}\pi)^2(u), \ldots, v_i,
				D_{\mathfrak{p}}(\pi(v_i)),\ldots)\\
				f' &=(D_{\mathfrak{p}'}\pi'(u), (D_{\mathfrak{p}'}\pi')^2(u), \ldots),
				\end{align*}
				where $v_i$ is the first half edge of $\nu$ that appears in $f_j$. Since $\pi(z)=\pi'(z)$
				if $z\not\in \nu$, we have $D_{\mathfrak{p}}\pi(u)=D_{\mathfrak{p}}\pi'(u)$, whence the
				entire subsequence from $D_{\mathfrak{p}}(\pi(u))$ to $v_i$ in $f_j$
				appears also in $f'$. In particular we have $v_i\in H(f')$, which means that $f'$ is
				incident to $\nu$ and the Claim follows.\qed
				
				Let $\epsilon$ be an embedding of the graph $G$ and $\nu$ be a vertex of $G$,
				where $\nu$ is incident to $q$ faces in $\epsilon$. Assume $\epsilon$ is represented
				by $\mathfrak{p}=(s,\pi)_k$.
				Similar to the situation of one-face maps,  we can define the localization at $\nu$
				which is a plane permutation having $q$ cycles, $(s_{\nu},\pi_{\nu})_q$, and that is
				obtained as follows:
				the $q$ cycles of $(s_{\nu},\pi_{\nu})_q$ are obtained from the $q$ cycles of $\mathfrak{p}$
				incident to $\nu$ by deleting all columns having no half edges of $\nu$.
				Let $D_{\nu}$ denote the diagonal of $(s_{\nu},\pi_{\nu})_q$.
				By construction, we have $s_{\nu}=D_{\nu}\circ \pi_{\nu}$, having $q$ cycles.
				
				Given a plane permutation $(s'_{\nu},\pi'_{\nu})_{q'}$, where $(D'_{\nu}, Par_{\pi'_{\nu}})=
				(D_{\nu},Par_{\pi_{\nu}})$, we can inflate w.r.t. $\mathfrak{p}$ into an embedding
				of $G$ as in the case of cyclic plane permutations. Namely, we substitute each
				diagonal-pair with the corresponding diagonal block in $\mathfrak{p}$ and then
				add any $\mathfrak{p}$-cycles containing half edges not incident to $\nu$.
				
				Fix an embedding $\epsilon$, represented by the plane permutation $(s,\pi)_{k}$,
				of genus $g(\epsilon)$. We compute in the following the distribution of genera
				resulting from reembedding the vertex $\nu$.
				Let $R_{\nu}(\Delta g)$ denote the number of different embeddings, $\epsilon'$, coming from reembedding $\nu$ such that $g(\epsilon')=g(\epsilon)+\Delta g$ and denote the cycle-type of
				$D_{\nu}$ as $\lambda(D_{\nu})$. Then we have
				
				\begin{theorem}\label{4thm1}
					\begin{align}
					R_{\nu}(\Delta g)= p^{\lambda(D_{\nu})}_{q+2\Delta g}(deg(\nu)),
					\end{align}
				\end{theorem}
				\proof
				Let $(s,\pi)_{k}$ represent $\epsilon$ and $(s',\pi')_{k+2\Delta g}$ represent $\epsilon'$,
				respectively. Here the index $k+2\Delta g$ stems from $g(\epsilon')=g(\epsilon)+\Delta g$
				which implies that $\epsilon'$ differs by $2\Delta g$ faces from $\epsilon$.
				
				According to Lemma~\ref{4lem1}, we have the following situation:
				$\bigcup_i H(f_i)$ is reorganized into $q+ 2\Delta g$ $\epsilon'$-faces,
				$f'_1,\ldots, f'_{q+2 \Delta g}$ and any other $\epsilon'$-face coincides with some
				$\epsilon$-face not incident to $\nu$.
				
				Let $(s'_{\nu}, \pi'_{\nu})_{q+2\Delta g}$ be the localization of $(s',\pi')_{k+2\Delta g}$,
				having the diagonal $D'_{\nu}$. By definition, $s'_{\nu}=D'_{\nu}\circ \pi'_{\nu}$ has
				$(q+2\Delta g)$ cycles.
				
				{\it Claim $1$.} Given $\epsilon$ represented by $(s,\pi)_{k}$, any reembedding of $\nu$,
				$\epsilon'$ represented by $(s',\pi')_{k+2\Delta g}$ satisfies $D_{\nu}'=D_{\nu}$.\\
				
				Suppose $\nu$ is incident to $q$ $\epsilon$-faces, $f_1,\ldots f_q$. Furthermore, suppose
				the $\epsilon'$-cycle of face $f'_i$ reads:
				\begin{align*}
				\left(\begin{array}{ccccccccccc}
				\boxed{v'_{i1}} & x_1 & \cdots & v'_{i2} & x_2 & \cdots & v'_{i3} &\cdots & v'_{it_i} & \cdots & y\\
				v'_{ij_1} & \cdots & x'_1 & v'_{ij_2} &\cdots & x'_2 & v'_{ij_3} & \cdots & v'_{ij_{t_i}} &\cdots & \boxed{z}
				\end{array}\right),
				\end{align*}
			where $v'_{ik}, v'_{ij_k} \in \nu \ \wedge \ v'_{ij_k}=\pi'_{\nu}(v'_{ik})$.	Then, by the same argument as in the proof for the Lemma~\ref{2lem2}, the diagonal block
				\begin{align*}
				\begin{array}{cccc}
				& x_l & \cdots & v'_{i(l+1)} \\
				v'_{ij_l} & \cdots & x'_l &
				\end{array}
				\end{align*}
				is also a diagonal block in $\epsilon$, which in turn implies $D_{\nu}'=D_{\nu}$.\\

				{\it Claim $2$.} Suppose $(s'_{\nu}, \pi'_{\nu})_{q+2\Delta g}$ is a localization
				such that $D_{\nu}'=D_{\nu}$ and $C(\pi'_{\nu})=1$. Then $(s'_{\nu}, \pi'_{\nu})_{q+2\Delta g}$
				can be inflated into an embedding $\epsilon'$ such that $g(\epsilon')=g(\epsilon)+\Delta g$,
				holds.\\
				
				Suppose $(s'_{\nu}, \pi'_{\nu})_{q+2\Delta g}$ is given by:
				\begin{align*}
				\left(\begin{array}{cccccccc}
				\boxed{v'_{11}} & \cdots & v'_{1t_1} &  \cdots & \boxed{v'_{(q+2\Delta g)1}} &
				\cdots & v'_{(q+2\Delta g)t_{q+2\Delta g}}\\
				\pi'_{\nu}({v'_{11}}) & \cdots & \boxed{\pi'_{\nu}(v'_{1t_1})} &
				\cdots & {\pi'_{\nu}(v'_{(q+2\Delta g)1})} & \cdots &
				\boxed{\pi'_{\nu}(v'_{(q+2\Delta g)t_{q+2\Delta g}})}
				\end{array}\right).
				\end{align*}
				Inflating every diagonal-pair into a diagonal block w.r.t.~$\epsilon$ and adding
				the $\epsilon$-cycles which are not incident to $\nu$, we obtain an embedding
				$\epsilon'$ with $2\Delta g$ more faces than $\epsilon$, i.e.,
				$g(\epsilon')=g(\epsilon)+\Delta g$. By construction, $\epsilon$ and $\epsilon'$ only differ by
				cyclic rearrangement of the half edges around $\nu$. \qed

				We proceed by studying the values of $\Delta g$ in Theorem~\ref{4thm1}, i.e.~the
				set $\{k|p_k^{\lambda}(n)\neq 0\}$. According to \cite{chr-1} we have:
				\begin{proposition}\cite{chr-1}\label{2pro1}
					For a plane permutation $\mathfrak{p}=(s,\pi)$ on $[n]$, the sum of the numbers of
					cycles in $\pi$ and in $D_{\mathfrak{p}}$ is smaller than $n+2$,
					$$
					\max\{k|p_k^{\lambda}(n)\neq 0\}\leq n+1-\ell(\lambda).
					$$
				\end{proposition}
				
				Next we show that the maximum can be always achieved.
				
				\begin{proposition}\label{2pro2}
					Let $\lambda\vdash n$ and $n\geq 1$. Then,
					\begin{align}
					\max\{k|p_k^{\lambda}(n)\neq 0\}=n+1-\ell (\lambda).
					\end{align}
				\end{proposition}
				\begin{proof}
					For $n=1$, the assertion is clear, whence we can assume w.l.o.g.~$n\geq 2$.
					For any permutation $\alpha$ on $[n]$ of cycle type $\lambda$ and $\ell(\lambda)=1$,
					we have $\alpha=\alpha e_n$ where $e_n$ is the identity permutation on $[n]$ which
					obviously has $n$ cycles. Therefore, in case of $\ell(\lambda)=1$,
					$
					\max\{k|p_k^{\lambda}(n)\neq 0\}=n=n+1-\ell (\lambda).	
					$
					Suppose for any $\lambda$ with $1\leq \ell(\lambda)=m<n$ holds
					$$
					\max\{k|p_k^{\lambda}(n)\neq 0\}=n+1-m.	
					$$
					Let $\alpha'$ be a permutation on $[n]$ of cycle type $\lambda'$ and
					$\ell(\lambda')=m+1$. Since $m+1\geq 2$, we can always find $a$ and $b$ such that
					$a$ and $b$ are in different cycles of $\alpha'$. Let $\alpha=\alpha'(a,b)$. Thus,
					$\alpha$ must be of cycle type $\mu$ for some $\mu$ such that $\ell(\mu)=m$.
					By assumption, there exists a relation
					$\alpha=s\pi$ such that $s$ has only one cycle and $\pi$ has $n+1-m$ cycles.
					Then,
					$$
					\alpha'=\alpha(a,b)=s\pi(a,b).
					$$
					Note $\pi(a,b)$ has the number of cycles either $n+1-m-1$ or $n+1-m+1$. The
					latter is impossible because it would contradict the bound established in
					Proposition~\ref{2pro1}. Hence, for any $\lambda'$ with $\ell(\lambda')=m+1$,
					$$
					\max\{k|p_k^{\lambda'}(n)\neq 0\}=n+1-m-1=n+1-\ell(\lambda'),
					$$
					which completes the proof of the proposition.
				\end{proof}

				\begin{corollary}\label{4cor2}
					Let $\epsilon$ be a fixed embedding of the graph $G$. Then for any vertex
					$\nu$ with localization $(s_{\nu}, \pi_{\nu})_q$, there exists for any
					$$
					-\lfloor\frac{deg(\nu)+1-\ell(\lambda(D_{\nu}))-q}{2}\rfloor \leq \Delta g \leq
					\lfloor \frac{q-1}{2}\rfloor.
					$$
					an embedding
					$\epsilon'$ of $G$ such that $g(\epsilon')=g(\epsilon)+\Delta g$.
				\end{corollary}
				\proof
				According to Corollary~\ref{2cor1}, we have $p^{\lambda}_k(n)\neq 0$ as long as
				$p^{\lambda}_{k+2i}(n)\neq 0$ for some $i>0$ holds.
				Furthermore, Proposition~\ref{2pro2} implies
				$p^{\lambda}_{deg(\nu)+1-\ell(\lambda(D_{\nu}))}(deg(\nu))\neq 0$.
				Therefore, for any
				$$
				1\leq d \leq deg(\nu)+1-\ell(\lambda(D_{\nu})), \quad d\equiv q \pmod{2},
				$$
				reembedding $\nu$ can lead to an embedding where $\nu$ is incident to
				$d$ faces.
				Accordingly, Euler's characteristic formula, implies
				$$
				-\lfloor\frac{deg(\nu)+1-\ell(\lambda(D_{\nu}))-q}{2}\rfloor \leq \Delta g \leq \lfloor \frac{q-1}{2}\rfloor,
				$$
				completing the proof of the corollary.\qed
				
				This result is similar to the result in~\cite{duke}, where it was shown for any $g_{min}(G)\leq
				g\leq g_{max}(G)$, there exists an embedding of $G$ on $S_g$.
				However, while it is very hard to obtain $g_{min}$ and $g_{max}$,
				we obtain easily the local minimum and the local maximum.

				Suppose we are given two vertices, such that there exists no face of $\epsilon$
				incident to both, then we call these two vertices $\epsilon$-face disjoint.
				In view of Lemma~\ref{4lem1}, Corollary~\ref{4cor2} has the following implication.

				\begin{corollary}\label{4cor22}
					Let $\epsilon$ be an embedding of the graph $G$.
					If the vertices $\nu_i=(s_{\nu_i}, \pi_{\nu_i})_{q_i}$, $1\leq i \leq m$, are mutually
					$\epsilon$-face disjoint,
					then there exists an embedding $\epsilon'$ of $G$ for any
					$$
					\sum_{i=1}^m -\lfloor\frac{deg(\nu_i)+1-\ell(\lambda(D_{\nu_i}))-q_i}{2}\rfloor \leq
					\Delta g \leq \sum_{i=1}^m \lfloor \frac{q_i-1}{2}\rfloor.
					$$
					such that $g(\epsilon')=  g(\epsilon)+\Delta g$.
				\end{corollary}

				The following corollary provides a necessary condition for an embedding of $G$ to be 
				of maximum genus as well as an easy-to-check necessary condition for an embedding of $G$ to be 
				of minimum genus.
				
				\begin{corollary}\label{4cor3}
					If $\epsilon$ is an embedding of the graph $G$ with genus $g_{max}(G)$, then every vertex is
					incident to at most $2$ faces in $\epsilon$.
					Furthermore, if $\epsilon$ is an embedding of the graph $G$ with genus $g_{min}(G)$,
					and $(s_{\nu}, \pi_{\nu})_{q_{\nu}}$ is a localization at $\nu$, then
					\begin{align}
					\ell(\lambda(D_{\nu}))+q_v=deg(\nu)+1.
					\end{align}
				\end{corollary}
				\proof The assertions are implied by Corollary~\ref{4cor2}.\qed

				The fact that if there exists a vertex incident to at least $3$ faces in an embedding,
				an embedding with higher genus always exists, is well known, see e.g., in~\cite{martin,xuong2}.
				However, to the best of our knowledge, given an embedding $\epsilon$, there is no
				simple characterization in order to determine if there exists an embedding of lower genus.
				Corollary~\ref{4cor3} gives a sufficient condition, i.e., 
				if $\ell(\lambda(D_{\nu}))+q_v\neq deg(\nu)+1$ for some vertex $v$,
				then there exists an embeding of lower genus.

\end{document}